\documentclass[12pt, reqno]{amsart}
\usepackage{amsmath, amsthm, amscd, amsfonts, amssymb, graphicx, color}
\usepackage[bookmarksnumbered, colorlinks, plainpages]{hyperref}
\usepackage{float}
\usepackage{graphicx,epstopdf}
\usepackage{bbm}
\usepackage{mathrsfs}
\usepackage{dsfont}
\usepackage{amssymb}
\usepackage{csquotes}
\usepackage[dvips]{epsfig}
\usepackage{latexsym}
\usepackage{exscale}
\usepackage[latin1]{inputenc}
\usepackage{setspace}

\usepackage{tikz}
\usepackage{tikz-cd}
\usetikzlibrary{matrix,arrows,decorations.pathmorphing}
\hypersetup{colorlinks=true,linkcolor=red, anchorcolor=green, citecolor=cyan, urlcolor=red, filecolor=magenta, pdftoolbar=true}

\newtheorem{theorem}{Theorem}[section]

\newtheorem{corollary}[theorem]{Corollary}
\theoremstyle{definition}
\newtheorem{definition}[theorem]{Definition}
\newtheorem{example}[theorem]{Example}

\theoremstyle{remark}
\newtheorem{remark}[theorem]{Remark}
\numberwithin{equation}{section}
\begin{document}

\setcounter{page}{1}

\title[Terzioglu Thm]{On a theorem of Terz\. {I}o\u{g}lu}
\begin{center}
\author[A. AKSOY ]{Asuman G\"{u}ven AKSOY}
\end{center}

\address{$^{*}$Department of Mathematics, Claremont McKenna College, 850 Columbia Avenue, Claremont, CA  91711, USA.}
\email{\textcolor[rgb]{0.00,0.00,0.84}{aaksoy@cmc.edu}}

%\address{$^{1}$Institute of Mathematical Sciences, Claremont Graduate University, 710 N. College Avenue, Claremont, CA  91711, USA.}
%\email{\textcolor[rgb]{0.00,0.00,0.84}{qidi.peng@cgu.edu}}

%\dedicatory{This paper is dedicated to Professor ABCD}

\subjclass[2010]{Primary 47B07; Secondary 47B06}

\keywords{Compact Operators, Approximation Schemes.}

%\date{Received: xxxxxx; Revised: yyyyyy; Accepted: zzzzzz.
%\newline \indent $^{*}$Corresponding author}

\begin{abstract}
The theory of compact linear operators acting on a Banach space has such a classical core and is familiar to many. Perhaps lesser known is the characterization theorem of Terzio\u{g}lu for compact maps. In this paper we consider Terzio\u{g}lu's theorem and its consequences.  We also  give a similar characterization theorem in case where  there is an approximation scheme on the Banach space.
\end{abstract} \maketitle

\section{Introduction}
Let $X$ and $Y$ be Banach spaces and $T:X \to Y$ be an operator. We say $T$ is compact if and only if  it maps closed unit ball  $B_X$ of $X$ into a pre-compact subset of $Y$.  In other words, $T$ is compact if and only if for every norm bounded sequence $\{x_n\}$ of $X$, the sequence $\{Tx_n\}$ has a norm convergent subsequence in $Y$. Equivalently, $T$ is compact if and only if for every $\epsilon >0$, there exists elements  $y_1,y_2, \dots,y_n \in Y$ such that $$ T(B_{X}) \subseteq \bigcup_{k=1}^{n} \{y_k+\epsilon B_Y\}$$ where by $B_X$  and $B_Y$ we mean the closed unit balls of $X$ and $Y$ respectively. Every compact linear operator is bounded, hence continuous, but clearly not every bounded linear map is compact since one can take the identity operator on an infinite dimensional space $X$.  Compact operators are natural generalizations of finite rank operators and thus dealing with compact operators provides us with the closest analogy to the usual theorems of finite dimensional spaces.  Recall that  $\mathcal{L}(X,Y)$  denotes the normed vector space of all continuous operators from $X$ to $Y$ and $\mathcal{L}(X)$ stands for $\mathcal{L}(X,X)$ and $\mathcal{K}(X,Y)$ is the collection of all compact operators from $X$ to $Y$.  It is well known that if  $Y$ is a Hilbert space  then any compact  $T:X \to Y$ is a limit of finite rank operators, in other words if $\mathcal{F}(X,Y)$ denotes the class of finite rank maps then, 
$$  \mathcal{K}(X,Y)= \overline{\mathcal{F}(X,Y)}$$  where the closure is taken in the operator norm. However, the situation is quite different for Banach spaces, not every operator  between Banach spaces is a uniform  limit of finite rank maps. For further information we refer the reader to a well known example due to P. Enflo  \cite{Enflo}, in which Enflo constructs a Banach space without the approximation property. The following classical results on compact operators will be used  for our  discussion later. 
\begin{theorem}
For Banach spaces $X$, $Y$ and $Z$, we have the following:
\begin{enumerate}
\item $\mathcal{K}(X,Y)$ is a norm closed vector  subspace of $\mathcal{L}(X,Y)$ .
\item If $
X\xrightarrow[]{~~~S~~~}Y\xrightarrow[]{~~~T~~~}Z 
$ are continuous operators and either $S$ or $T$ is compact, then $TS$  is likewise compact.

\end{enumerate}
\end{theorem} 
If one consider the continuos operators on a Banach space $X$, the above theorem asserts the fact that compact operators on $X$ form a two sided ideal in $\mathcal{L}(X)$. The following theorem of Schauder simply states that an operator is compact if and only if its adjoint is compact.
\begin{theorem} [Schauder]
A norm bounded operator $T: X \to Y$ between Banach spaces is compact if and only if its adjoint $T^*: Y^* \to X^*$ is compact.
\end{theorem}
The main idea in proving Schauder's theorem lies in the fact that $$ ||P_n T -T|| \to 0 \quad \mbox{implies} \quad || T^* P_n -T^*|| \to 0 $$ where $P_n : X \to \mbox{span} \{e_1,\dots, e_n\}$. A well known  proof of Schauder's theorem may be found in Yosida [\cite{Yos}, p.282].

For our discussion below we also need the following characterization of the compact sets in a Banach space; in some sense, it is a comment on the smallness of compact sets. 
\begin{theorem}[Grothendieck] \label{thm:Groth}
A subset of a Banach space is compact  if and only if  it is included in the closed convex hull of a sequence that converges in norm to zero.
\end{theorem} In other words, if we have $K$ a compact subset of a Banach space $X$, then we can find a sequence $\{x_n\}$ in $X$ such that $$ ||x_n|| \to 0 \quad \mbox{and} \quad K \subseteq \overline{\mbox{co}}\{x_n\}.$$ For a proof we refer the reader to[ \cite{Dis} , p.3]

%%%%%%%%%%%%%%%
 \section{ Terz\.{i}o\u{g}lu's Theorem}
\begin{theorem}[Terzio\u{g}lu \cite{Ter1}] \label{thm:Tosun}
An operator $T:X \to Y$ between two Banach spaces is compact if and only if there exists a sequence $\{u_n\}$  of linear functionals in $X^*$ with $|| u_n || \to 0$ such that the inequality 
$$ ||Tx|| \leq \sup_{n} \left | < u_n, x> \right | $$ holds for every $x\in X$.
\end{theorem}
\begin{proof}
Suppose $T: X \to Y$ is compact, then by Schauder's theorem $T^*: Y^* \to X^*$ is compact; thus by definition, if $V$ denotes the closed unit ball of $Y^*, $ $T^*(V)$ is a norm totally bounded subset of $X^*$. Now applying Grothendieck's result, we have a sequence $\{u_n\}$ of elements of $X^*$ with $||u_n|| \to 0$ and $T^* (V) \subseteq \overline{\mbox{co}\{u_n\}}$. In other words, each element of $T^*(V)$ can be written  of the form 
$$ \sum_{n=1}^{\infty} \alpha_n u_n \quad \mbox{with}\quad \sum_{n=1}^{\infty} |\alpha_n | \leq 1.$$  Thus, for each $x\in X$ we have
$$||Tx|| =\sup_{||v|| \leq 1} \left| <T^*v, x>\right| \leq \sum_{n=1}^{\infty} |\alpha_n| \sup _{n} | < u_n, x>|.$$
Suppose $T$ satisfies the inequality $ ||Tx|| \leq \sup_{n} \left | < u_n, x> \right | $ for some sequence $\{u_n\} \in X^* $. For  $\epsilon >0$ choose $N$ such that $||u_n|| < \epsilon$ for $n> N$ and set $$M_{\epsilon} = \{ x\in X:\,\, < u_i,x> =0  \quad \mbox{for}\,\, i=1,2,\dots N  \}, $$ then one can have 

 $$T^* (\mathring{V}) \subset  \epsilon  \mathring{U} + M_{\epsilon}^{\bot}$$ where $U$ denotes the unit ball of $X$ and  for each linear subspace $M$ of $X$, the polar of $M$ denoted by  $\mathring{M}$ is a linear subspace of $X^*$ defined as:
$$ \mathring{M}:= \left\{ a\in X^*:\quad  |<x,a>|\leq 1 \quad \mbox{for}\quad x\in M\right\},$$

 this shows that $T^*$ is compact  and hence $T$ is compact.

\end{proof}

 An application of the Theorem \ref{thm:Tosun}  yields that every compact mapping of a Banach space into a $\mathcal{P}_{\lambda}$-space is $\infty$-nuclear.
\begin{definition}
We say $X$ is a  $\mathcal{P}_{\lambda}$ space, ($\lambda \geq 1$) if every bounded linear operator $T$ from a Banach space $Y $ to $X$ and every $Z \supset Y$ there is a linear extension $\tilde {T}$ of $ Z $ to $X$ with $$||\tilde{T} ||\leq \lambda ||T||.$$ 
\end{definition}
As illustrated in the following diagram:
\[
\begin{tikzcd}
&Z \arrow{rd}{\overset{\sim}{T}}\\
&Y \arrow[hook]{u} \arrow{r}{}[swap]{T}
&X
\end{tikzcd}
\]

 If $||\tilde{T} ||= ||T||$ in the above definition, we call $X$ extendible.  This property is related to  the existence of a global Hahn-Banach type an extension.  J. Lindenstrauss in \cite{Lin1} examines the problem  when is the extension $\tilde{T}$ is compact if  $T$ itself is compact and the author's results are diverse and numerous and touches upon many related topics. 
 Next, we define \textit{infinite nuclear} mappings, this concept was first introduced in \cite{Pit2}.
 \begin{definition}
 Let $X $ and $Y$ be Banach spaces and $T:X  \to Y $ a linear operator. Then $T $ is said to be infinite-nuclear, if there are sequences $\{u_n\} \subset  X^*$ and $\{y_n\} \subset Y$ such that $\displaystyle \lim_{n} ||u_n|| = 0$, 
$$\displaystyle \sup_{||v|| \leq 1}\left \{\displaystyle \sum_{n=1}^{\infty} |v(y_n)|:\,\, v \in Y^*\right \} < +\infty$$ and  $$T x = \displaystyle \sum_{n=1}^{\infty}< u_n, x>y_n $$ 
  for $ x \in X$. 
 \end{definition}
  As an application to Terzio\u{g}lu's Theorem, under the condition that $T: X \to Y$ where $Y$ is a $\mathcal{P}_{\lambda}$-space, Terzio\u{g}lu also  obtains a precise expression for $Tx$, which we state in the following:
  \begin{theorem}[\cite{Ter1}]
  Let $T$ be a compact mapping of a Banach space $X$ into a $\mathcal{P}_{\lambda}$ space $Y$.Then for every $\epsilon>0$ there exists sequence $\{u_n\}$ in $ X^*$ with $$\displaystyle\lim_n ||u_n|| =0 \quad \mbox{ and}\quad  \displaystyle \sup_{n} ||u_n|| \leq ||T|| + \epsilon$$ and a sequence $\{y_n\} \in Y$ with $\displaystyle \sup_{||v|| \leq 1} \displaystyle \sum_{n=1}^{\infty} |< v, y_n> | \leq \lambda $ such that $T$ has the form
 $$ Tx= \sum_{n=1}^{\infty} <u_n, x> y_n.$$
  
  \end{theorem} 
  The complete details of the proof can be found in \cite{Ter1}. However, it is worth pointing out that idea of the proof  provides a factorization of a compact map through the space $c_0$ as follows:\\[.02in]
 Use Theorem \ref{thm:Tosun},  choose the sequence  $\{u_n\}$ in $ X^*$ satisfying 
$$\displaystyle\lim_n ||u_n|| =0\,\,\,\mbox{ and}\,\,\,  \displaystyle \sup_{n} ||u_n|| \leq ||T|| + \epsilon\,\,\,\mbox{and}\,\, ||Tx|| \leq \sup|<u_n, x>|.$$ 
Define linear mapping $$S: X \to c_0 \quad \mbox{ by} \quad Sx=\{<u_n,x>\}, $$ and observe that  $S$ is compact, then define a linear mapping $$R_0: S(X) \to Y\,\,\mbox{by}\,\, R_0(Sx)=Tx,$$  the inequality
$$ ||R_0(Sx)||=||Tx|| \leq  \sup|<u_n, x>|= ||Sx||$$  implies that $||R_0|| \leq 1$.  

\[
\begin{tikzcd}
&c_0 \arrow{rd}{\overset{\sim} {R}}\\
&S(E) \arrow[hook]{u} \arrow{r}{}[swap]{R_0}
&F
\end{tikzcd}
\]

Since $Y$ is a $\mathcal{P}_{\lambda}$ space, there exists an extension $\tilde{R}$ of $R_0$ from $\tilde{R:}  c_0\to F$ with $||\tilde{R}|| \leq \lambda ||R_0||=\lambda$ and

$$
E\xrightarrow[]{~~~S~~~}c_0\xrightarrow[]{~~~\widetilde R~~~}F, 
$$
evidently $T=\widetilde{R} S$.

 By considering $\{e_n\}\in c_0$ and setting $y_n =\widetilde{R}(e_n)$ we obtain
$$ \displaystyle \sup_{||v|| \leq 1} \displaystyle  \sum_{n=1}^{\infty} |< v, y_n> | \leq \lambda,\,\,\mbox{and}\,\, Tx= \sum_{n=1}^{\infty} <u_n, x> y_n.$$

Using all of the above results of Terzio\u{g}lu  one can find the following conclusions in \cite{Ter2}.
\begin{corollary}
\begin{enumerate}
\item Every $\mathcal{P}_{\lambda}$ space has the approximation property.
\item Every compact linear operator of an $L^{\infty}$ space into a Banach space is infinite-nuclear.
\item Let $T$  be a compact linear map of an infinite-dimensional space $X$ into a Banach space $Y$. Then there exists an infinite dimensional closed subspace $M$ of $X$ such that
 $T_{M}: M \to T(M)$ is infinite nuclear.
\end{enumerate}
\end{corollary}

\section {Compactness with Approximation Scheme}
Approximation schemes were introduced by Butzer and Scherer for Banach spaces in 1968  \cite{But} and later by Brudnij and Krugljak \cite{BK}. These concepts find its best application in a paper by Pietsch  \cite{Pi}, where  he defined approximation spaces, proved embedding, reiteration and representation results and established connection to interpolation spaces.

Let $X$ be a Banach space and $\{A_n\}$ be a sequence of subsets of $X$ satisfying:
\begin{enumerate}
\item $A_1 \subseteq A_2 \subseteq \dots \subseteq X$
\item $\lambda A_n \subseteq A_n$ for all scalars $\lambda$ and $n=1,2,\dots$.
\item $A_m + A_n \subseteq A_{m+n}$ for $m,n= 1,2,\dots$.
\end{enumerate}

For example, for $1\leq p < \infty$ if we consider the space  $X= L_p[0,1]$, then the collection of sets  $\{A_n\}= \{L_{p+ \frac{1}{n}}\}$ form an approximation scheme like above. \\ Pietsch's approximation spaces $ X^{\rho}_{\mu}$   ($0< \rho< \infty, \,\, 0< \mu \leq \infty $)  is defined by considering the $n$-th approximation number $\alpha_n(f,X)$, where
$$ \alpha_n(f,X): = \inf \{ ||f-a|| : \,\, a\in A_{n-1}\}$$  and 
$$ X^{\rho}_{\mu}=\{ f\in X: \,\,\{n^{\rho-\frac{1} {\mu} } \, \alpha_n(f,X)\}\in \ell^{\mu} \}.$$
In the same paper \cite{Pi}, embeddings, composition and commutation as well as representation interpolation of such spaces are studied and applications to the distribution of Fourier coefficients and eigenvalues of integral operators are given.

 In the following we consider for each $n\in \mathbb{N}$ a family of subsets $Q_n$ of $X$ satisfying the very same  three conditions stated above. For example  for $Q_n$ could be the set of all at most $n$-dimensional subspaces of any Banach space $X$, or if our Banach space $X= \mathcal{L}(E)$, namely the set of all bounded linear operators on another Banach space $E$, then we can take $Q_n= N_n(E)$ the set of all $n$-nuclear maps on $E$.

Compactness relative to an approximation scheme for bounded sets and linear operators can be studied by using Kolmogorov diameters as follows.
%\begin{definition}
Let $D \subset X$ be a bounded subset and $U_X$ denote the closed unit ball of $X$. Suppose $Q=(Q_n(X)_{n\in \mathbb{N}})$ be an approximation scheme on $X$, then the $n$th Kolmogorov diameter of $D$ with respect to this scheme $Q$ is denoted by $\delta_n(D, Q)$ and defined as
$$ \delta_n(D, Q)= \inf\{r>0 : \,\, D \subset rU_X +A\quad \mbox{for some} \quad A \in Q_n(X)\}.$$

Let $Y$ be another Banach space and $T\in \mathcal{L}(Y,X)$, then the $n$th Kolmogorov diameter of $T$ with respect to this scheme $Q$ is denoted by $\delta_n(T, Q)$ and defined as
$$ \delta_n(T, Q)= \delta_n(T(U_X), Q) .$$

\begin{definition} 
We say $D$ is $Q$-compact set   if $$ \lim_{n}\delta_n(D, Q)=0$$ and similarly $T\in \mathcal{L}(Y,X)$ is a $Q$-compact map, if $$ \lim_{n} \delta_n(T, Q)=0.$$

\end{definition}
The following example illustrates that  not every  $Q$-compact operator is compact.
\begin{example} 

%%%%%%%%%%%%%
Let $\{r_n(t)\}$ be the space spanned by the Rademacher functions.  It can be seen from the Khinchin inequality \cite{Lin}that
\begin{equation}
\ell^2 \approx  \{r_n(t)\}\subset L_p[0,1] \text{ for all }1\leq p \leq \infty.
\end{equation}
We define an approximation scheme $A_n$ on $L_p[0,1]$ as follows:
\begin{equation}
 A_n=L_{p+\frac{1}{n}}.
\end{equation}
$L_{p+\frac{1}{n}}\subset L_{p+\frac{1}{n+1}}$ gives us $A_n\subset A_{n+1}$. for $n=1,2,\dots,$ and it is easily seen that $A_n+A_m \subset A_{n+m}$ for $n,m=1,2,\dots,$ and that $\lambda A_n \subset A_n$ for all $\lambda$.  Thus $\{A_n\}$ is an approximation scheme.

%%%%%%%%%%%
Next, we claim that for $p\geq 2$ the projection $P: L_p[0,1] \to R_p$ is a $Q$-compact map, but not compact,
%%%%%%%%%%%%%
where $R_p$ denotes the closure of the span of $\{r_n(t)\}$ in $L_p[0,1]$.  
\[
\begin{array}{ccc}
    L_p & \stackrel{i}{\longrightarrow} & L_2  \\
P \downarrow &   & \downarrow P_2 \\
 R_p & \stackrel{j}{\longleftarrow}& R_2 \\
\end{array}
\]

We know that for $p\geq 2$, $L_p[0,1]\subset L_2[0,1]$ and $R_2$ is a closed subspace of $L_2[0,1]$ and  $$P=j\circ P_2\circ i$$ where $i,j$ are isomorphisms shown in the above figure. $P$ is not a compact operator, because  dim$R_p=\infty$, on the other hand it is a $Q$-compact operator because, if we let $U_{R_p},U_{L_p}$ denote the closed unit balls of $R_p$ and $L_p$ respectively, it is easily seen that $P(U_{L_p})\subset \|P\|U_{R_p}$.  But $U_{R_p} \subset CU_{R_{P+\frac{1}{n}}}$ where $C$ is a constant follows from the Khinchin inequality.  Therefore, $$P(U_{L_p})\subset L_{p+\frac{1}{n}},\quad\mbox{ which gives} \quad \delta_n(P,Q)\to 0.$$ 

Next we give a characterization of $Q$-compact sets as subsets of the closed convex hull of certain uniform null-sequences.
\begin{definition}
Suppose $X$ is a Banach space with an approximation scheme $Q_n$.  A sequence $\{x_{n,k\}_k}$  in $X$ is called an  order $c_0$-sequence if
\begin{enumerate}
\item $\forall n=1,2,\dots$ there exists $A_n \in Q_n$ and a sequence $\{x_{n,k}\}_k \subset A_n$
\item $||x_{n,k}|| \to 0$ as $n\to \infty$ uniformly in $k$.
\end{enumerate}
\end{definition}  
\begin{theorem} \label{thm:order}
Let $X$ be a Banach space with an approximation scheme with sets $A_n \in Q_n$
 satisfy the condition $ |\lambda | A_n \subset A_n$ for $|\lambda| \leq 1$. A bounded subset $D$ of $X$ is $Q$-compact if and only if there is an order $c_0$-sequence  $\{x_{n,k\}_k} \subset A_n$ such that

 $$D \subset\left \{ \sum_{n=1}^{\infty} \lambda_n x_{n,k(n)} :\quad  x_{n,k(n)}\in (x_{n,k})\quad \sum_{n=1}^{\infty} |\lambda_n| \leq 1 \right \}.$$
  
 \end{theorem}
 Proof of the above theorem can be obtained from the one given for  $p$-Banach spaces in  \cite{Ak-Nk}. Clearly this is an analogue of Grothendieck's theorem given above in Theorem \ref{thm:Groth}   for $ Q$-compact sets.
\section{Terzio\u{g}lu's Theorem for $Q$-compact Maps}
Terzio\u{g}lu's  characterization of compact maps relies on both Grothendieck's  and Schauder's theorems.  Above Theorem 3.2 is Grothendieck's theorem for $Q$-compact sets, therefore we turn our attention to the  relationship between $T$ being $Q$-compact and its transpose $T^*$ being $Q$-compact. The relationship between the approximation numbers of $T$ and $T^*$ was studied by several authors, it is shown in \cite{Hut} that  for $T \in \mathcal{L}(X)$, we have
$$ \mbox{dist}(T, \mathcal{F}) \leq 3\, \mbox{dist}(T^*, \mathcal{F^*})$$
where $ \mathcal{F}$ and $ \mathcal{F^*}$ denote the class of all finite rank operators on $X$ and $X^*$ respectively. Central to the proof of  such result is the assumption of local reflexivity possessed by all Banach spaces, (see \cite{Lin}).  It is not hard to show that   if we assume that our space $X$ with approximation scheme $Q_n$ satisfies slight modification of this property, called extended local reflexivity principle, then we have  
$$   \alpha_n(T, Q) \leq 3\, \alpha_n(T^*, Q^*).$$  Where, by $\alpha_n(T, Q)$ we mean the $n$th approximation number defined as
$$ \alpha_n(T, Q) = \inf \left\{||T-B||: \,\, B\in \mathcal{L}(X),\,\, B(X) \in Q_n (X)\right\}.$$

However, we do not have a proof of the  Schauder's theorem for $Q$-compact maps. In the following we present a result analogous to Terzio\u{g}lu's Theorem  for $Q$-compact maps under the assumption that  both $T$ and $T^*$ are $Q$-compact.
\begin{theorem}
Let $E$ and $F$ be Banach spaces , $T\in \mathcal{L}(E,F) $ and  assume that  both $T$ and $T^*$ are $Q$-compact maps. Then there exists sequence $\{u_{n,k}\}\in Q_n$ with $||u_{n,k}|| \to 0$ for $n\to \infty$ for  all $k$, such that the inequality
$$ ||Tx|| \leq \mbox{sup} | < u_{n,k(n)}, x>|$$ holds for every $x\in E$.  Here $Q_n$ is a ``special" class of subsets of $E^*$ with the property that  $u_{n,k(n)}\in \{u_{n,k}\}$.
\end{theorem}
\begin{proof}
Since $T^*: F^* \to E^*$  is $Q$-compact, thus by the Theorem \ref{thm:order},  $T^*(U_{F^*})$ is a $Q$-compact set, thus there exists a sequence $\{u_{n,k}\}_k \subset A_n\in Q_n$ such that  $|| u_{n,k}|| \to 0$ as $n\to \infty$ uniformly in $k$ and  
 $$T^*(U_{F^*}) \subset\left \{ \sum_{n=1}^{\infty} \lambda_n u_{n,k(n)} :\quad  u_{n,k(n)}\in (u_{n,k})\quad \sum_{n=1}^{\infty} |\lambda_n| \leq 1 \right \}.$$
Then for each $x\in E$, we have:

$$ || Tx|| =\displaystyle \sup _{v \in U_{F^*} } |<v, Tx> | =  \sup _{v \in U_{F^*} } |<T^*v, x> | = \sup _{n} |<\sum_{n=1}^{\infty}\lambda_n u_{n,k(n)}, x > |$$
and thus 
$$   || Tx|| \leq \sum_{n=1}^{\infty} |\lambda_n| \sup_{n} |<  u_{n,k(n)}, x>|  \leq  \,\sup_{n} | < u_{n,k(n)}, x>|.$$
\end{proof}

\begin{remark}
We say a map  $T\in \mathcal{L}(E,F)$ is a $Q$-compact map, if $ \lim_{n} \delta_n(T, Q)=0$.  To obtain "Schauder's Theorem"  for $Q$-compact maps, one seeks  a relationship  between $\delta_n(T)$ and $\delta_n(T^*)$. K. Astala  in \cite{As} proved  that under the  assumption  that the Banach space  $E$ has the lifting property and the Banach space  $F$ has the extension property, for a map   $T\in \mathcal{L}(E,F)$, one has  $\gamma(T) =\gamma (T^*)$, where  $\gamma(T)$ denotes the measure of non-compactness.  Since 
$$ \lim_{n\to \infty}\delta_n(T) = \gamma(T),$$  by imposing extension and lifting properties on $E$ and $F$ respectively and keeping tract of approximation schemes on these spaces one might obtain Schauder's type of a theorem in this special case.

\end{remark}

%%%%%%%%%%%%
\end{example}

 \bibliographystyle{amsplain}

\end{document}